\documentclass[11pt]{amsart}
\usepackage{enumerate} 

\newtheorem{theorem}{Theorem}[section] 
\newtheorem{lemma}[theorem]{Lemma}

\newtheorem{corollary}[theorem]{Corollary}
\theoremstyle{definition}

\theoremstyle{remark}
\newtheorem{remark}[theorem]{Remark}
\numberwithin{equation}{section}

\DeclareMathOperator{\tr}{tr}

\title[Near-isometric duality of Hardy norms]{Near-isometric duality of Hardy norms with applications to harmonic mappings}

\author{Leonid V. Kovalev}
\address{215 Carnegie, Mathematics Department, Syracuse University, Syracuse, NY 13244, USA}
\email{lvkovale@syr.edu}
\thanks{L.V.K. supported by the National Science Foundation grant DMS-1764266.}

\author{Xuerui Yang}
\address{215 Carnegie, Mathematics Department, Syracuse University, Syracuse, NY 13244, USA}
\email{xyang20@syr.edu}
\thanks{X.Y. supported by Young Research Fellow award from Syracuse University.}



\begin{document}
\baselineskip5.7mm

\maketitle

\begin{abstract} Hardy spaces in the complex plane and in higher dimensions have natural finite-dimensional subspaces formed by polynomials or by linear maps. We use the restriction of Hardy norms to such subspaces to describe the set of possible derivatives of harmonic self-maps of a ball, providing a version of the Schwarz lemma for harmonic maps. These restricted Hardy norms display unexpected near-isometric duality between the exponents 1 and 4, which we use to give an explicit form of harmonic Schwarz lemma.  
\end{abstract}

\section{Introduction}

This paper connects two seemingly distant subjects: the geometry of Hardy norms on finite-dimensional spaces and the gradient of a harmonic map of the unit ball. Specifically, writing $H^1_*$ for the dual of the Hardy norm $H^1$ on complex-linear functions (defined in~\S\ref{sec:polynomials}), we obtain the following description of the possible gradients of harmonic maps of the unit disk $\mathbb D$.  

\begin{theorem}\label{thm-schwarz-intro} A vector $(\alpha, \beta)\in \mathbb C^2$ is the Wirtinger derivative at $0$ of some harmonic map $f\colon \mathbb D\to\mathbb D$ if and only if $\|(\alpha, \beta)\|_{H^1_*} \le 1$. 
\end{theorem}

Theorem~\ref{thm-schwarz-intro} can be compared to the behavior of holomorphic maps $f\colon \mathbb D\to\mathbb D$ for which the set of all possible values of $f'(0)$ is simply $\overline{\mathbb D}$. The appearance of $H^1_*$ norm here leads one to look for a concrete description of this norm. It is well known that the duality of holomorphic Hardy spaces $H^p$ is not isometric, and in particular the dual of $H^1$ norm is quite different from $H^\infty$ norm even on finite dimensional subspaces (see~\eqref{not-infty}). However, it has a striking similarity to $H^4$ norm. 

\begin{theorem}\label{thm-norm-comp-intro} For all $\xi\in \mathbb C^2\setminus \{(0,0)\}$, 
$1\le  \|\xi\|_{H^1_*} / \|\xi\|_{H^4} \le 1.01$.
\end{theorem}

Since the $H^4$ norm can be expressed as $\|(\xi_1, \xi_2)\|_4 = (|\xi_1|^4 + 4|\xi_1 \xi_2|^2 + |\xi_4|^4)^{1/4}$,  Theorem~\ref{thm-norm-comp-intro} supplements  Theorem~\ref{thm-schwarz-intro} with an explicit estimate. 

In general, Hardy norms are merely quasinorms when $p<1$, as the triangle inequality fails. However, their restrictions to the subspaces of degree $1$ complex polynomials or of $2\times 2$ real matrices are actual norms  (Theorem~\ref{thm-hardy-norm} and Corollary~\ref{cor-hardy-norm}). We do not know if this property holds for $n\times n$ matrices with $n > 2$. 

The paper is organized as follows. Section~\ref{sec:polynomials} introduces Hardy norms on polynomials. In Section~\ref{sec:dual} we prove  Theorem~\ref{thm-norm-comp-intro}. Section~\ref{sec:schwarz} concerns the Schwarz lemma for planar harmonic maps,  Theorem~\ref{thm-schwarz-intro}. In section~\ref{sec:higher} we consider higher dimensional analogues of these results. 

\section{Hardy norms on polynomials}\label{sec:polynomials}

For a polynomial $f\in \mathbb C[z]$, the Hardy space ($H^p$) quasinorm is defined by 
\[
\|f\|_{H^p} = \left(
\frac{1}{2\pi} \int_0^{2\pi} |f(e^{it})|^p\,dt
\right)^{1/p}
\]
where $0<p<\infty$. There are two limiting cases: $p\to \infty$ yields the supremum norm 
\[
\|f\|_{H^\infty} = \max_{t\in \mathbb R} |f(e^{it})|
\]
and the limit $p\to 0$ yields the Mahler measure of $f$: 
\[
\|f\|_{H^0} = \exp\left(
\frac{1}{2\pi} \int_0^{2\pi} \log |f(e^{it})|\,dt
\right).
\]
An overview of the properties of these quasinorms can be found 
in~\cite[Chapter 13]{RahmanSchmeisser} and in~\cite{Pritsker}. In general they satisfy the definition of a norm only when $p\ge 1$. 

The Hardy quasinorms on vector spaces $\mathbb C^n$ are defined by  
\[
\|(a_1, \dots, a_{n})\|_{H^p}
= \|f\|_{H^p}, \quad 
f(z) = \sum_{k=1}^{n} a_{k} z^{k-1}.
\]
We will focus on the case $n=2$, which corresponds to the $H^p$ quasinorm of degree $1$ polynomials $a_1+a_2z$. These quantities appear as multiplicative constants in sharp inequalities for polynomials of general degree: see Theorems~13.2.12 and 14.6.5 in~\cite{RahmanSchmeisser}, or Theorem~5 in~\cite{Pritsker}. In general, $H^p$ quasinorms cannot be expressed in elementary functions even on $\mathbb C^2$. Notable exceptions include 
\begin{equation}\label{explicit-norms}
\begin{split}
 \|(a_1, a_2)\|_{H^0} & = \max(|a_1|, |a_2|),  \\
 \|(a_1, a_2)\|_{H^2} & = \left(|a_1|^2+|a_2|^2\right)^{1/2},  \\
 \|(a_1, a_2)\|_{H^4} & = \left(|a_1|^4 + 4|a_1|^2|a_2|^2 + |a_2|^4\right)^{1/4},  \\ 
 \|(a_1, a_2)\|_{H^\infty} & = |a_1| + |a_2|.
 \end{split}
\end{equation}
Another easy evaluation is
\begin{equation}\label{specialH1}
\|(1, 1)\|_{H^1}    
= \frac{1}{2\pi}\int_0^{2\pi}|1+e^{it}|\,dt
% = \frac{1}{2\pi}\int_0^{2\pi}\sqrt{2+2\cos t}\,dt \\
 = \frac{1}{2\pi}\int_0^{2\pi}2|\cos(t/2)|\,dt = \frac{4}{\pi}.
\end{equation}
However, the general formula for the $H^1$ norm on $\mathbb C^2$ involves the complete elliptic integral of the second kind $E$. Indeed, writing $k = |a_2/a_1|$, we have 
\begin{equation}\label{elliptic-integral}
\begin{split}
\|(a_1, a_2)\|_{H^1}
& = |a_1| \, \|(1, k)\|_{H^1}
= \frac{|a_1|}{2\pi}\int_0^{2\pi} |1+ke^{2it}|\,dt 
\\ = & |a_1| \frac{2(k+1)}{\pi} \int_0^{\pi/2} \sqrt{1-\left(\frac{2\sqrt{k}}{k+1}\right)^2 \sin^2 t}\,dt
\\= & |a_1| \frac{2(k+1)}{\pi} E\left(\frac{2\sqrt{k}}{k+1}\right).
\end{split}
\end{equation}

Perhaps surprisingly, the Hardy quasinorm on $\mathbb C^2$ is a norm (i.e., it satisfies the triangle inequality) even when $p<1$.  

\begin{theorem}\label{thm-hardy-norm} The Hardy quasinorm on $\mathbb C^2$ is a norm for all $0\le p\le \infty$. In addition, it has the symmetry properties 
\begin{equation}\label{symmetry}
\|(a_1, a_2)\|_{H^p} = \|(a_2, a_1)\|_{H^p} 
= \|(|a_1|, |a_2|)\|_{H^p}. 
\end{equation}
\end{theorem}

\begin{proof} For $p=0,\infty$ all these statements follow from~\eqref{explicit-norms}, so we assume $0<p<\infty$. The identities 
\begin{equation}\label{swap}
\int_0^{2\pi} |a_1 + a_2 e^{it}|^p\,dt
= \int_0^{2\pi} |a_1e^{-it} + a_2|^p\,dt
= \int_0^{2\pi} |a_2 + a_1e^{it}|^p\,dt
\end{equation}
imply the first part of~\eqref{symmetry}. Furthermore, the first integral in~\eqref{swap} is independent of the argument of $a_2$ while the last integral is independent of the argument of $a_1$. This completes the proof of~\eqref{symmetry}.

It remains to prove the triangle inequality in the case $0<p<1$. To this end, consider the following function of $\lambda\in \mathbb R$.
\begin{equation}\label{Gp}
G(\lambda) := \|(1, \lambda)\|_{H^p}
=  \left( \frac{1}{2\pi} \int_0^{2\pi} |1 + \lambda e^{it}|^p\,dt   \right)^{1/p}.
\end{equation}
We claim that $G$ is convex on $\mathbb R$. If $|\lambda|<1$, the binomial series 
\[
(1 + \lambda e^{it})^{p/2} = \sum_{n=0}^\infty \binom{p/2}{n} \lambda^n e^{nit}
\]
together with Parseval's identity imply
\begin{equation}\label{Gpseries}
G(\lambda) = \left(
\sum_{n=0}^\infty \binom{p/2}{n}^2 \lambda^{2n}
\right)^{1/p}.
\end{equation}
Since every term of the series is a convex function of $\lambda$, it follows that $G$ is convex on $[-1, 1]$. The power series also shows that $G$ is $C^\infty$ smooth on $(0,  1)$. For $\lambda>1$ the symmetry property~\eqref{symmetry} yields
$G(\lambda) = \lambda G(1/\lambda)$ 
which is a convex function by virtue of the identity $G''(\lambda) = \lambda^{-3} G''(1/\lambda)$. 
The piecewise convexity of $G$ on $[0, 1]$ and $[1, \infty)$ will imply its convexity on $[0, \infty)$ (hence on $\mathbb R$) as soon as we show that $G$ is differentiable at $\lambda=1$. Note that $|1+\lambda e^{it}|^p$ is differentiable with respect  to $\lambda$ when $e^{it}\ne -1$ and that for $\lambda$ close to $1$, 
\begin{equation}\label{dominate}
\frac{\partial}{\partial \lambda} |1+\lambda e^{it}|^p  \le p |1+\lambda e^{it}|^{p-1}) \le  C|t-\pi|^{p-1} 
\end{equation}
for all $t\in [0, 2\pi]\setminus\{\pi\}$, with $C$ independent of $\lambda, t$. The integrability of the right hand side of~\eqref{dominate} justifies differentiation under the integral sign:  
\[
\frac{d}{d\lambda} G(\lambda)^p = 
\frac{1}{2\pi} \int_0^{2\pi} \frac{\partial}{\partial \lambda} |1 + \lambda e^{it}|^p\,dt.  
\]
Thus $G'(1)$ exists. 

Now that $G$ is known to be convex, the convexity of the function $F(x, y) := \|(x, y)\|_{H^p} = xG(y/x)$ on the halfplane $(x, y)\in \mathbb R^2$, $x>0$, follows by computing its Hessian, which exists when $|y|\ne x$:  
\[
H_F = G''(y/x) \begin{pmatrix}
x^{-3}y^2  & -x^{-2}y  \\ 
-x^{-2}y  & x^{-1}   \\ 
\end{pmatrix}.
\]
Since $H_F$ is positive semidefinite, and $F$ is $C^1$ smooth even on the lines $|y|=|x|$, the function $F$ is convex on the halfplane $x>0$. By symmetry, convexity holds on other coordinate halfplanes as well, and thus on all of $\mathbb R^2$. The fact that $G$ is an increasing function on $[0, \infty)$ also shows that $F$ is an increasing function of each of its variables in the first quadrant $x, y\ge 0$. 

Finally, for any two points $(a_1, a_2)$ and $(b_1, b_2)$ in $\mathbb C^2$ we have
\[ \begin{split}
\|(a_1+b_1, a_2+b_2)\|_{H^p} &= 
F( |a_1+b_1|, |a_2+b_2|) 
\le F( |a_1|+|b_1|, |a_2|+|b_2|)
\\
&\le F(|a_1|, |a_2|) + F(|b_1|, |b_2|)
= \|(a_1, a_2)\|_{H^p} + \|(b_1, b_2)\|_{H^p} 
\end{split} \]
using~\eqref{symmetry} and the monotonicity and convexity of $F$. 
\end{proof}

\begin{remark} In view of Theorem~\ref{thm-hardy-norm} one might guess that the restriction of $H^p$ quasinorm to the polynomials of degree at most $n$ should satisfy the triangle inequality provided that $p > p_n$ for some $p_n<1$. This is not so: the triangle inequality fails for any $p<1$ even when the quasinorm is restricted to quadratic polynomials. Indeed, for small $\lambda\in \mathbb R$ we have 
\[
\begin{split}
\|(\lambda, 1, \lambda)\|_{H^p}^p
& = \frac{1}{2\pi}\int_0^{2\pi}(1 + 2\lambda \cos t)^p \,dt
\\ 
& = \frac{1}{2\pi}\int_0^{2\pi}\left(1 + 2\lambda p \cos t
+ 2\lambda^2 p(p-1) \cos^2 t + O(\lambda^3)\right) \,dt \\
& = 1 + \lambda^2 p(p-1) + O(\lambda^3) 
\end{split}
\]
and this quantity has a strict local maximum at $\lambda=0$ provided that $0<p<1$. 
\end{remark}

\section{Dual Hardy norms on polynomials}\label{sec:dual}

The space $\mathbb C^n$ is equipped with the inner product 
$\langle \xi, \eta \rangle = 
\sum_{k=1}^n \xi_k \overline{\eta_k}$.
Let $H^p_*$ be the norm on $\mathbb C^n$ dual to $H^p$, that is 
\begin{equation}\label{def-dual-norm}
\|\xi\|_{H^p_*} = \sup\left\{
|\langle \xi, \eta \rangle| \colon \|\eta\|_{H^p} \le 1 \right\}  
= \sup_{\eta\in\mathbb C^n\setminus\{0\}} 
\frac{|\langle \xi, \eta \rangle|}{\|\eta\|_{H^p}}.
\end{equation}
One cannot expect the $H^p_*$ norm to agree with $H^q$ for $q=p/(p-1)$ (unless $p=2$), as the duality of Hardy spaces is not isometric~\cite[Section 7.2]{Duren}. However, on the space $\mathbb C^2$ the $H^1_*$ norm turns out to be surprisingly close to $H^4$, indicating that $H^1$ and $H^4$ have nearly isometric duality in this setting. The following is a restatement of Theorem~\ref{thm-norm-comp-intro} in the form that is convenient for the proof.  

\begin{theorem}\label{thm-norm-comp} For all $\xi\in \mathbb C^2$ we have \begin{equation}\label{norm-comp1}
\|\xi\|_{H^1} \le \|\xi\|_{H^4_*} \le 1.01 \|\xi\|_{H^1}    
\end{equation}
and consequently 
\begin{equation}\label{norm-comp2}
\|\xi\|_{H^4} \le \|\xi\|_{H^1_*} \le 1.01 \|\xi\|_{H^4}.    
\end{equation}
\end{theorem}

It should be noted that while the $H^1$ norm on $\mathbb C^2$ is a non-elementary function~\eqref{elliptic-integral}, the $H^4$ norm has a simple algebraic form~\eqref{explicit-norms}. To see that having the exponent $p=4$, rather than the expected $p=\infty$, is essential in Theorem~\ref{thm-norm-comp}, compare the following:
\begin{equation}\label{not-infty}
\begin{split}
\|(1, 1)\|_{H^1_*} & = \frac{2}{\|(1, 1)\|_{H^1}} = \frac{\pi}{2}
\approx 1.57
, \\ 
\|(1, 1)\|_{H^\infty} & = 2, \\ 
\|(1, 1)\|_{H^4} & = 6^{1/4} \approx 1.57.
\end{split}
\end{equation}

The proof of Theorem~\ref{thm-norm-comp} requires an elementary lemma from analytic geometry. 

\begin{lemma}\label{lem-circle} If $0<r<a$ and $b\in \mathbb R$, then 
\begin{equation}\label{sup-circle}
\sup_{\theta\in \mathbb R} 
\frac{b - r\sin \theta}{a  - r\cos\theta}
= \frac{ab + r\sqrt{a^2+b^2-r^2}}{a^2-r^2}.
\end{equation}
\end{lemma}

\begin{proof} The quantity being maximized is the slope of a line through $(a, b)$ and a point on the circle $x^2+y^2=r^2$. The slope is maximized by one of two tangent lines to the circle passing through $(a, b)$. Let $\tan \alpha = b/a$ be the slope of the line $L$ through $(0, 0)$ and $(a, b)$. This line makes angle $\beta$ with the tangents, where $\tan \beta = r/\sqrt{a^2+b^2-r^2}$. Thus, the slope of the tangent of interest is   
\[
\tan(\alpha+\beta) = 
\frac{\tan\alpha + \tan \beta}{1-\tan \alpha \tan \beta} = \frac{b\sqrt{a^2+b^2-r^2} +ar }{a\sqrt{a^2+b^2-r^2} - br}
\]
which simplifies to~\eqref{sup-circle}.
\end{proof}

\begin{proof}[Proof of Theorem~\ref{thm-norm-comp}] 
Because of the symmetry properties~\eqref{symmetry} and the homogeneity of norms, it suffices to 
consider $\xi = (1, \lambda)$ with $0\le \lambda \le 1$. This restriction on $\lambda$ will remain in force throughout this proof. 

The function 
\[
G(\lambda) := \|(1,\lambda)\|_{H^1} 
= \frac{1}{2\pi}\int_0^{2\pi} |1+\lambda e^{it}|\,dt  
\]
has been intensely studied due to its relation with the arclength of the ellipse and the complete elliptic integral~\cite{AB, BPR2}. It can be written as 
\begin{equation}\label{G-def}
G(\lambda) = 
\frac{L(x, y)}{\pi(x+y)}
= {}_2F_1(-1/2,-1/2; 1;\lambda^2) 
= \sum_{n=0}^\infty \left(\frac{(-1/2)_n}{n!}\right)^2 \lambda^{2n}
\end{equation}
where $L$ is the length of the ellipse with semi-axes $x, y$ and $\lambda=(x-y)/(x+y)$. The Pochhammer symbol $(z)_n = z(z+1)\cdots (z+n-1)$ and the hypergeometric function ${}_2F_1$ are involved in~\eqref{G-def} as well. A direct way to obtain the Taylor series~\eqref{G-def} for $G$ is to use the binomial series as in~\eqref{Gpseries}.

As noted in~\eqref{explicit-norms}, the $H^4$ norm of $(1,\lambda)$ is an elementary function: 
\[
F(\lambda)  := \|(1,\lambda)\|_{H^4}
= (1+4\lambda^2+\lambda^4)^{1/4}.
\]
The dual norm $H^4_*$ can be expressed as
\begin{equation}\label{F-def}
F^*(\lambda) :=
\|(1,\lambda)\|_{H^4_*} = 
\sup_{t\in \mathbb R} \frac{1 + \lambda t}{(1+4t^2+t^4)^{1/4}}
\end{equation}
where the second equality follows from~\eqref{def-dual-norm} by letting $b=(1, t)$. 
Similarly, the $H^1_*$ norm of $(1, \lambda)$ is  
\begin{equation}\label{Gs-def}
G^*(\lambda) :=
\|(1,\lambda)\|_{H^1_*} = 
\sup_{t\in \mathbb R} \frac{1 + \lambda t}{G(t)}.
\end{equation}

Our first goal is to prove that
\begin{equation}\label{001}
G^*(\lambda) \le 1.01 F(\lambda).
\end{equation}
The proof of~\eqref{001} is based on Ramanujan's approximation $G(\lambda)\approx 3-\sqrt{4-\lambda^2}$ which originally appeared in~\cite{Ramanujan}; see~\cite{AB} for a discussion of the history of this and several other approximations to $G$. Barnard, Pearce, and Richards~\cite[Proposition 2.3]{BPR2} proved that Ramanujan's approximation gives a lower bound for $G$: 
\begin{equation}\label{BPR}
G(\lambda) \ge  3-\sqrt{4-\lambda^2}.
\end{equation}
We will use this estimate to obtain an upper bound for $G^*$.

The supremum in~\eqref{Gs-def} only needs to be taken over $t\ge 0$ since the denominator is an even function. Furthermore, it can be restricted  to $t\in [0, 1]$ because for $t>1$ the homogeneity and symmetry properties of $H^1$ norm imply 
\[
\frac{1+\lambda t}{\|(1, t)\|_{H^1}}
= \frac{t^{-1} +\lambda}{\|(1, t^{-1})\|_{H^1}} < 
\frac{1 + \lambda t^{-1}}{\|(1, t^{-1})\|_{H^1}}. 
\]
Restricting $t$ to $[0,1]$ in~\eqref{Gs-def} allows us to use inequality~\eqref{BPR}:
\begin{equation}\label{Gs-upper}
G^*(\lambda) \le \sup_{t\in [0, 1]} \frac{1 + \lambda t}{3-\sqrt{4-t^2}}.
\end{equation}
Writing $t = - 2\sin \theta$ and applying Lemma~\ref{sup-circle} we obtain 
\begin{equation}\label{Gs-upper2}
\begin{split}
G^*(\lambda) & \le \lambda 
\sup_{\theta\in [-\pi/6, 0]} \frac{\lambda^{-1} - 2\sin \theta}{3 - 2\cos \theta}
\le \lambda\, \frac{3\lambda^{-1} + 2 \sqrt{5 + \lambda^{-2}}}{5}
\\ & = \frac{3 + 2\sqrt{1+5\lambda^2}}{5}.
\end{split}
\end{equation}

The function
\[
f(s) := \frac{3 + 2\sqrt{1+5s}}{(1+4s+s^2)^{1/4}}
\] 
is increasing on $[0, 1]$. Indeed, 
\[\begin{split}
f'(s) = \frac{3(6s+2 - (s+2)\sqrt{1+5s})}
{2\sqrt{1+5s} (1+4s+s^2)^{5/4}}
\end{split} \]
which is positive on $(0, 1)$ because 
\[
(6s+2)^2 - (s+2)^2(1+5s) = 5s^2(3-s)> 0.
\]
Since $f$ is increasing, the estimate~\eqref{Gs-upper2} implies
\[
\frac{G^*(\lambda)}{F(\lambda)} \le \frac{1}{5} f(\lambda^2)
\le \frac{1}{5}f(1) = \frac{3+2\sqrt{6}}{5\cdot 6^{1/4}} < 1.01.
\]
This completes the proof of~\eqref{001}. 

Our second goal is the following comparison of $F^*$ and $G$ with a polynomial: 
\begin{equation}\label{two-sides}
G(\lambda) \le 1 + \frac{1}{4}\lambda^2 
+ \frac{1}{64}\lambda^4 
+ \frac{1}{128}\lambda^6 \le F^*(\lambda).
\end{equation}
To prove the left hand side of~\eqref{two-sides}, let $T_4(\lambda) = 1+\lambda^2/4 + \lambda^4/64$ be the Taylor polynomial of $G$ of degree $4$. Since all Taylor coefficients of $G$ are nonnegative~\eqref{G-def}, the function 
\[
\phi(\lambda) := \frac{G(\lambda) - T_4(\lambda)}{\lambda^6} - \frac{1}{128}
\]
is increasing on $(0, 1]$. At $\lambda = 1$, in view of~\eqref{specialH1}, it evaluates to 
\[
G(1) - 1 - \frac{1}{4} - \frac{1}{64} - \frac{1}{128}
 = \frac{4}{\pi} - \frac{163}{128}  
\]
which is negative because $512/163 = 3.1411\ldots < \pi$. Thus $\phi(\lambda)<0$ for $0<\lambda\le 1$, proving the left hand side of~\eqref{two-sides}. 

The right hand side of~\eqref{two-sides} amounts to the claim that for every $\lambda$ there exists $t\in \mathbb R$ such that 
\[
\frac{1+\lambda t}{(1+4t^2+t^4)^{1/4}} 
\ge 
1 + \frac{1}{4}\lambda^2 + \frac{1}{64}\lambda^4 + \frac{1}{128}\lambda^6. 
\]
This is equivalent to proving that the polynomial
\[
\Phi(\lambda, t) := (1+\lambda t)^4 - (1+4t^2+t^4)\left(1 + \frac{1}{4}\lambda^2 + \frac{1}{64}\lambda^4 + \frac{1}{128}\lambda^6 \right)^4 
\]
satisfies $\Phi(\lambda, t) \ge 0$ for some $t$ depending on $\lambda$. We will do so by choosing $t = 4\lambda/(8 - 3\lambda^2)$. The function
\[
\Psi(\lambda) := (8-3\lambda^2)^4 \Phi(\lambda, 4\lambda/(8 - 3\lambda^2))
\]
is a polynomial in $\lambda $ with rational coefficients. Specifically, 
\begin{equation}\label{monster}
\begin{split}
\frac{\Psi(\lambda)}{\lambda^8} & = 50 + \lambda^2 - 
\frac{149}{2^4}\lambda^4 - \frac{209}{2^6} \lambda^6 - \frac{5375}{2^{12}}\lambda^8
- \frac{3069}{2^{13}} \lambda^{10} - \frac{8963}{2^{17}}\lambda^{12}  \\
& - \frac{7837}{2^{19}}\lambda^{14} - \frac{36209}{2^{24}}\lambda^{16}
- \frac{2049}{2^{23}}\lambda^{18} - \frac{1331}{2^{25}}\lambda^{20} 
- \frac{45}{2^{25}}\lambda^{22} - \frac{81}{2^{28}}\lambda^{24}
\end{split}
\end{equation}
which any computer algebra system will readily confirm. On the right hand side of~\eqref{monster}, the coefficients of $\lambda^4, \lambda^6, \lambda^8$ are less than $10$ in absolute value, while the coefficients of higher powers are less than $1$ in absolute value. Thanks to the constant term of $50$, the expression~\eqref{monster} is positive as long as $0<\lambda \le 1 $. This completes the proof of~\eqref{two-sides}.  

In conclusion, we have $G(\lambda)\le F^*(\lambda)$ from~\eqref{two-sides} and $G^*(\lambda)\le 1.01 F(\lambda)$ from~\eqref{001}. This proves the first half of~\eqref{norm-comp1} and the second half of~\eqref{norm-comp2}. The other parts of~\eqref{norm-comp1}--\eqref{norm-comp2} follow by duality. 
\end{proof}

\section{Schwarz lemma for harmonic maps}\label{sec:schwarz} 

Let $\mathbb D = \{z\in \mathbb C\colon |z|<1\}$ be the unit disk in the complex plane. The classical Schwarz lemma concerns holomorphic maps $f\colon \mathbb D\to \mathbb D$ normalized by $f(0)=0$. It asserts in part that  $|f'(0)| \le 1$ for such maps. This inequality is best possible in the sense that for any complex number $\alpha$ such that $|\alpha|\le 1$ there exists $f$ as above with $f'(0)=\alpha$. Indeed, $f(z)=\alpha z$ works.

The story of the Schwarz lemma for harmonic maps $f\colon \mathbb D\to \mathbb D$, still normalized by $f(0)=0$, is more complicated. Such maps satisfy the Laplace equation $\partial \bar \partial f=0$ written here in terms of Wirtinger's derivatives
\[
\partial f = \frac{1}{2}\left(\frac{\partial f}{\partial x} - i\frac{\partial f}{\partial y}\right),\quad 
\bar \partial f = \frac{1}{2}\left(\frac{\partial f}{\partial x} + i\frac{\partial f}{\partial y}\right).
\]
The estimate $|f(z)|\le \frac{4}{\pi}\tan^{-1}|z|$ (see~\cite{Heinz} or~\cite[p. 77]{Duren-harmonic}) implies that 
\begin{equation}\label{sch1}
|\partial f(0)| + |\bar \partial f(0)| \le \frac{4}{\pi}.
\end{equation}
Numerous generalizations and refinements of the harmonic Schwarz lemma appeared in recent years~\cite{KalajVuorinen, Mateljevic}. An important difference with the holomorphic case is that~\eqref{sch1} does not completely describe the possible values of the derivative $(\partial f(0), \bar \partial f(0))$. Indeed, an application of Parseval's identity shows that 
\begin{equation}\label{sch2}
|\partial f(0)|^2 + |\bar \partial f(0)|^2 \le 1
\end{equation}
and neither of~\eqref{sch1} and ~\eqref{sch2} imply each other. It turns out that the complete description of possible derivatives at $0$ requires the dual Hardy norm from~\eqref{def-dual-norm}. The following is a refined form of Theorem~\ref{thm-schwarz-intro} from the introduction. 

\begin{theorem}\label{thm-schwarz} For a vector $(\alpha, \beta)\in \mathbb C^2$ the following are equivalent: 
\begin{enumerate}[(i)]
    \item\label{s1} there exists a harmonic map $f\colon \mathbb D\to\mathbb D$ with $f(0)=0$, $\partial f(0)=\alpha $, and $\bar\partial f(0)=\beta$;
    \item\label{s2} there exists a harmonic map $f\colon \mathbb D\to\mathbb D$ with $\partial f(0)=\alpha $ and $\bar\partial f(0)=\beta$;
    \item\label{s3} $\|(\alpha, \beta)\|_{H^1_*} \le 1$. 
\end{enumerate}
\end{theorem}

\begin{remark} Both~\eqref{sch1} and~\eqref{sch2} easily follow from Theorem~\ref{thm-schwarz}. To obtain~\eqref{sch1}, use the definition of $H^1_*$ together with the fact that $\|(a_1, a_2)\|_{H^1} = 4/\pi$  whenever $|a_1|=|a_2|=1$ (see~\eqref{specialH1}, \eqref{symmetry}). To obtain~\eqref{sch2}, use the comparison of Hardy norms: $\|\cdot \|_{H^1}\le \|\cdot \|_{H^2}$, hence $\|\cdot \|_{H^1_*}\ge \|\cdot \|_{H^2_*} = \|\cdot \|_{H^2}$. 
\end{remark}

\begin{remark} Combining Theorem~\ref{thm-schwarz} with Theorem~\ref{thm-norm-comp} we obtain  
\begin{equation}\label{eqn-harmonic-Schwarz}
\|(\partial f(0), \bar \partial f(0))\|_{H^4} \le 1
\end{equation}
for any harmonic map $f\colon \mathbb D\to \mathbb D$. In view of~\eqref{explicit-norms} this means 
$|\partial f(0)|^4 + 4 |\partial f(0) \bar \partial f(0)|^2 + |\bar\partial f(0)|^4\le 1$.
\end{remark}

\begin{proof}[Proof of Theorem~\ref{thm-schwarz}] \eqref{s1}$\implies$\eqref{s2} is trivial. Suppose that~\eqref{s2} holds. To prove~\eqref{s3}, we must show that  
\begin{equation}\label{ps1}
|\alpha \bar \gamma + \beta\bar \delta | \le \|(\gamma, \delta)\|_{H^1}   
\end{equation}
for every vector $(\gamma, \delta)\in \mathbb C^2$. Let $g(z) = \gamma z + \delta \bar z$. Expanding $f$ into the Taylor series $f(z) = f(0) + \alpha z + \beta \bar z + \dots $ and using the orthogonality of monomials on every circle $|z|=r$, $0<r<1$, we obtain 
\begin{equation}\label{ps0}
|\alpha \bar \gamma + \beta\bar \delta| = 
\frac{1}{2\pi r^2}
\left|\int_0^{2\pi} f(re^{it})\overline{g(re^{it})}\,dt \right|
\le \frac{1}{2\pi r^2} \int_0^{2\pi} |g(re^{it})|\,dt.
\end{equation}
Letting $r\to 1$ and observing that
\begin{equation}\label{equiv-h1}
\frac{1}{2\pi}\int_0^{2\pi} |\gamma e^{it} + \delta e^{-it}|\,dt
= \frac{1}{2\pi}\int_0^{2\pi} |\gamma + \delta e^{-2it}|\,dt
= \frac{1}{2\pi}\int_0^{2\pi} |\gamma + \delta e^{it}|\,dt
= \|(\gamma, \delta)\|_{H^1}
\end{equation}
completes the proof of~\eqref{ps1}.

It remains to prove the implication \eqref{s3}$\implies$\eqref{s1}. Let $\mathcal F_0$ be the set of harmonic maps $f\colon \mathbb D\to\mathbb D$ such that $f(0)=0$, and let $\mathcal D = \{(\partial f(0), \bar\partial f(0))\colon f\in \mathcal F_0\}$. Since $\mathcal F_0$ is closed under convex combinations, the set $\mathcal D$ is convex. Since the function $f(z)=\alpha z + \beta \bar z$ belongs to $\mathcal F_0$ when $|\alpha|+|\beta|\le 1$, the point $(0, 0)$ is an interior point of $\mathcal D$. The estimate~\eqref{sch2} shows that $\mathcal D$ is bounded. 
Furthermore, $c \mathcal D\subset \mathcal D$ for any complex number $c$ with $|c|\le 1$,  because $\mathcal F_0$ has the same property. We claim that $\mathcal D$ is also a closed subset of $\mathbb C^2$. Indeed, suppose that a sequence of vectors $(\alpha_n, \beta_n)\in \mathcal D $ converges to $(\alpha, \beta)\in \mathbb C^2$. Pick a corresponding sequence of maps $f_n\in \mathcal F_0$. Being uniformly bounded, the maps $\{f_n\}$ form a normal family~\cite[Theorem 2.6]{ABR}. Hence there exists a subsequence $\{f_{n_k}\}$ which converges uniformly on compact subsets of $\mathbb D$. The limit of this subsequence is a map $f\in \mathcal F_0$ with  $\partial f(0)=\alpha$ and $\bar \partial f(0)=\beta$. 

The preceding paragraph shows that $\mathcal D$ is the closed unit ball for some norm $\|\cdot \|_{\mathcal D}$ on $\mathbb C^2$. The implication \eqref{s3}$\implies$\eqref{s1} amounts to the statement that $\|\cdot \|_{\mathcal D} \le \|\cdot \|_{H^1_*}$. We will prove it in the dual form \begin{equation}\label{ps2}
\sup\{|\gamma \overline{\alpha} + \delta \overline{\beta}| \colon (\alpha, \beta) \in \mathcal D \} \ge \|(\gamma, \delta) \|_{H^1} 
\quad \text{for all }  (\gamma, \delta)\in \mathbb C^2.
\end{equation}
Since norms are continuous functions, it suffices to consider $(\gamma, \delta)\in \mathbb C^2$ with $|\gamma|\ne |\delta|$. Let $g\colon \mathbb D\to \mathbb D$ be the harmonic map with boundary values 
\[
g(z) = \frac{\gamma z + \delta \bar z}{|\gamma z + \delta \bar z|},\quad |z|=1.
\]
Note that $g(-z)=-g(z)$ on the boundary, and therefore everywhere in $\mathbb D$. In particular, $g(0)=0$, which shows $g\in \mathcal F_0$. Let $(\alpha, \beta)= (\partial g(0), \bar\partial g(0))\in \mathcal D$. A computation similar to~\eqref{ps0} shows that
\[
\begin{split}
\gamma \bar \alpha + \delta\bar \beta & =
\frac{1}{2\pi}
\int_0^{2\pi} (\gamma e^{it} + \delta e^{-it}) \overline{g(e^{it})}
\,dt \\ 
& = 
\frac{1}{2\pi}
\int_0^{2\pi} (\gamma e^{it} + \delta e^{-it}) 
\frac{\overline{\gamma e^{it} + \delta e^{-it}}}{|\gamma e^{it} + \delta e^{-it}|}\,dt 
\\ & =
\frac{1}{2\pi}
\int_0^{2\pi} |\gamma e^{it} + \delta e^{-it}|\,dt =  \|(\gamma, \delta) \|_{H^1}
\end{split}
\]
where the last step uses~\eqref{equiv-h1}. This proves~\eqref{ps2} and completes the proof of Theorem~\ref{thm-schwarz}.
\end{proof}

\section{Higher dimensions}\label{sec:higher}

A version of the Schwarz lemma is also available for harmonic maps of the (Euclidean) unit ball $\mathbb B$ in $\mathbb R^n$. Let $\mathbb S = \partial \mathbb B$. For a square matrix $A\in \mathbb R^{n\times n}$, define  its Hardy quasinorm by 
\begin{equation}\label{def-matrix-norm}
\|A\|_{H^p} = \left(\int_{\mathbb S}\|Ax\|^p\,d\mu(x)\right)^{1/p}
\end{equation} 
where the integral is taken with respect to normalized surface measure $\mu$ on $\mathbb S$ and the vector norm $\|Ax\|$ is the Euclidean norm. In the limit $p\to \infty$ we recover the spectral norm of $A$, while the special case $p=2$ yields the Frobenius norm of $A$ divided by $\sqrt{n}$. The case $p=1$ corresponds to ``expected value norms'' studied by Howe and Johnson in~\cite{HoweJohnson}. Also, letting $p\to 0$ leads to  
\begin{equation}\label{Mahler-norm}
\|A\|_{H^0} = \exp\left(\int_{\mathbb S}\log \|Ax\| \,d\mu(x)\right)
\end{equation} 

In general, $H^p$ quasinorms on matrices are not submultiplicative. However, they have another desirable feature, which follows directly from~\eqref{def-matrix-norm}: $\|UAV\|_{H^p} = \|A\|_{H^p}$ for any orthogonal matrices $U, V$.  The singular value decomposition shows that $\|A\|_{H^p}= \|D\|_{H^p}$ where $D$ is the diagonal matrix with the singular values of $A$ on its diagonal. 

Let us consider the matrix inner product $\langle A, B\rangle = \frac{1}{n}\tr(B^T A)$, which is  normalized so that $\langle I, I\rangle = 1$.  This inner product can be expressed by an integral involving the standard inner product on $\mathbb R^n$ as follows: 
\begin{equation}\label{integral-inner-product}
\langle A, B\rangle = \int_{\mathbb S} \langle Ax, Bx\rangle \,d\mu(x). 
\end{equation}
Indeed, the right hand side of~\eqref{integral-inner-product} is the average of the numerical values $\langle B^TAx, x\rangle$, which is known to be the normalized trace of $B^TA$, see~\cite{Kania}.

The dual norms $H^p_*$ are defined on $\mathbb R^{n\times n}$ by 
\begin{equation}\label{def-dual-matrix-norm}
\|A\|_{H^p_*} = \sup\left\{
\langle A, B\rangle \colon 
\|B\|_{H^p} \le 1 \right\}  
= \sup_{B\in\mathbb R^{n\times n}\setminus\{0\}} 
\frac{\langle A, B\rangle}{\|B\|_{H^p}}.
\end{equation}

Applying H\"older's inequality to~\eqref{integral-inner-product} yields $\langle A, B\rangle \le \|A\|_{H^q}\|B\|_{H^p}$ when $p^{-1}+q^{-1}=1$. Hence $\|A\|_{H^p_*} \le \|A\|_{H^q}$ but in general the inequality is strict. As an exception, we have $\|A\|_{H^2_*} = \|A\|_{H^2}$ because 
$\langle A, A\rangle = \|A\|_{H^2}^2$. As in the case of polynomials, our interest in dual Hardy norms is driven by their relation to harmonic maps.  

\begin{theorem}\label{thm-schwarz-higher} For a matrix $A\in \mathbb R^{n\times n}$ the following are equivalent: 
\begin{enumerate}[(i)]
    \item\label{s1h} there exists a harmonic map $f\colon \mathbb B\to\mathbb B$ with $f(0)=0$ and $Df(0)=A$;
    \item\label{s2h} there exists a harmonic map $f\colon \mathbb B\to\mathbb B$ with $Df(0)=A$;
    \item\label{s3h} $\|A\|_{H^1_*} \le 1$. 
\end{enumerate}
\end{theorem}

\begin{proof} Since the proof is essentially the same as of Theorem~\ref{thm-schwarz}, we only highlight some notational differences. Suppose (\ref{s2h}) holds. Expand $f$ into a series of spherical harmonics,
$f(x) = \sum_{d=0}^\infty p_d(x)$ where $p_d\colon \mathbb R^n\to\mathbb R^n$ is a harmonic polynomial map that is homogeneous of degree $d$. Note that $p_1(x) = Ax$. For any $n\times n$ matrix $B$ the orthogonality of spherical harmonics~\cite[Proposition 5.9]{ABR} yields 
\[
\langle A, B\rangle 
= \lim_{r\nearrow 1} \int_{\mathbb S}\langle f(rx), Bx\rangle \,d\mu(x)
\le \|B\|_1
\]
which proves~(\ref{s3h}). 

The proof of (\ref{s3h})$\implies$(\ref{s1h}) is based on considering, for any nonsingular matrix $B$, a harmonic map $g\colon \mathbb B\to\mathbb B$ with boundary values $g(x) = (Bx)/\|Bx\|$. Its derivative $A=Dg(0)$ satisfies 
\[\begin{split} 
\langle B, A\rangle = 
\int_{\mathbb S}\langle Bx, g(x) \rangle \,d\mu(x)
= \int_{\mathbb S}\frac{\langle Bx, Bx \rangle}{\|Bx\|} \,d\mu(x) = \|B\|_{H^1}
\end{split}\]
and (\ref{s1h}) follows by the same duality argument as in Theorem~\ref{thm-schwarz}. 
\end{proof}

As an indication that the near-isometric duality of $H^1$ and $H^4$ norms (Theorem~\ref{thm-norm-comp}) may also hold in higher dimensions, we compute the relevant norms of $P_k$, the matrix of an orthogonal projection of rank $k$ in $\mathbb R^3$. For rank 1 projection, the norms are  
\[
\begin{split}
\|P_1\|_{H^1} & = 
\int_{0}^1 r\,dr = \frac{1}{2}, \\ 
\|P_1\|_{H^4} & = 
\left(\int_{0}^1 r^4\,dr\right)^{1/4} = \frac{1}{5^{1/4}} \approx 0.67, \\ 
\|P_1\|_{H^1_*} & = 
\frac{\langle P_1, P_1\rangle}{\|P_1\|_1}  = \frac{1/3}{1/2} = \frac{2}{3} \approx 0.67. 
\end{split}
\]
For rank 2 projection, they are 
\[
\begin{split}
\|P_2\|_{H^1} & = 
\int_{0}^1 \sqrt{1-r^2}\,dr = \frac{\pi}{4}, \\ 
\|P_2\|_{H^4} & = 
\left(\int_{0}^1 (1-r^2)^2\,dr\right)^{1/4} = \left(\frac{8}{15}\right)^{1/4}  \approx 0.85, \\ 
\|P_2\|_{H^1_*} & = 
\frac{\langle P_2, P_2\rangle}{\|P_2\|_1}  = \frac{2/3}{\pi/4} = \frac{8}{3\pi} \approx 0.85. 
\end{split}
\]
This numerical agreement does not appear to be merely a coincidence, as  numerical experiments with random $3\times 3$ indicate that the ratio $\|A\|_{H^1_*} / \|A\|_{H^4}$ is always near $1$. However, we do not have a proof of this. 

As in the case of polynomials, there is an explicit formula for the $H^4$ norm of matrices. Writing $\sigma_1,\dots, \sigma_n$ for the singular values of $A$, we find  
\begin{equation}\label{H4matrix}
\|A\|_{H^4}^4 = \alpha \sum_{k=1}^n \sigma_k^4  + 2 \beta \sum_{k<l} \sigma_k^2 \sigma_l^2 
\end{equation}
where $\alpha = \int_{\mathbb S} x_1^4\,d\mu(x)$ and $\beta = \int_{\mathbb S} x_1^2x_2^2\,d\mu(x)$. For example, if $n=3$, the expression~\eqref{H4matrix} evaluates to
\[
\|A\|_{H^4}^4 = \frac{1}{5} \sum_{k=1}^3 \sigma_k^4 + \frac{2}{15}\sum_{k<l} \sigma_k^2 \sigma_l^2.   
\]

Theorem~\ref{thm-hardy-norm} has a corollary for $2\times 2$ matrices.

\begin{corollary}\label{cor-hardy-norm}
The $H^p$ quasinorm on the space of $2\times 2$ matrices satisfies the triangle inequality even when $0\le p<1$. \end{corollary}

\begin{proof}
A real linear map  $x\mapsto Ax$ in $\mathbb R^2$ can be written in complex notation as $z\mapsto az+b\bar z$ for some $(a, b)\in \mathbb C^2$. A change of variable yields 
\[
\int_{|z|=1} |az+b\bar z|^p 
= \int_{|z|=1} |a+bz|^p
\]
which implies $\|A\|_{H^p} = \|(a, b)\|_{H^p}$ for $p>0$. The latter is a norm on $\mathbb C^2$ by  Theorem~\ref{thm-hardy-norm}. The case $p=0$ is treated in the same way. 
\end{proof}

The aforementioned relation between a $2\times 2$ matrix $A$ and a complex vector $(a, b)$ also shows that the singular values of $A$ are $\sigma_1 = |a|+|b|$ and $\sigma_2 = ||a|-|b||$. It then follows from~\eqref{explicit-norms} that 
\[
\|A\|_{H^0} = \max(|a|, |b|)
 = \frac{\sigma_1 + \sigma_2}{2},
\]
which is, up to scaling, the trace norm of $A$. Unfortunately, this relation breaks down in dimensions $n>2$: for example, rank $1$ projection $P_1$ in $\mathbb R^3$ has $\|P_1\|_{H^0} = 1/e$ while the average of its singular values is $1/3$. 

We do not know whether $H^p$ quasinorms with $0\le p<1$ satisfy the triangle inequality for $n\times n$ matrices when $n\ge 3$.   

\bibliography{references.bib} 

\begin{thebibliography}{10}

\bibitem{AB}
Gert Almkvist and Bruce Berndt.
\newblock Gauss, {L}anden, {R}amanujan, the arithmetic-geometric mean,
  ellipses, {$\pi$}, and the {\it {l}adies diary}.
\newblock {\em Amer. Math. Monthly}, 95(7):585--608, 1988.

\bibitem{ABR}
Sheldon Axler, Paul Bourdon, and Wade Ramey.
\newblock {\em Harmonic function theory}, volume 137 of {\em Graduate Texts in
  Mathematics}.
\newblock Springer-Verlag, New York, second edition, 2001.

\bibitem{BPR2}
Roger~W. Barnard, Kent Pearce, and Kendall~C. Richards.
\newblock A monotonicity property involving {$_3F_2$} and comparisons of the
  classical approximations of elliptical arc length.
\newblock {\em SIAM J. Math. Anal.}, 32(2):403--419, 2000.

\bibitem{Duren-harmonic}
Peter Duren.
\newblock {\em Harmonic mappings in the plane}, volume 156 of {\em Cambridge
  Tracts in Mathematics}.
\newblock Cambridge University Press, Cambridge, 2004.

\bibitem{Duren}
Peter~L. Duren.
\newblock {\em Theory of {$H^{p}$} spaces}.
\newblock Pure and Applied Mathematics, Vol. 38. Academic Press, New
  York-London, 1970.

\bibitem{Heinz}
Erhard Heinz.
\newblock On one-to-one harmonic mappings.
\newblock {\em Pacific J. Math.}, 9:101--105, 1959.

\bibitem{HoweJohnson}
Eric~C. Howe and Charles~R. Johnson.
\newblock Expected-value norms on matrices.
\newblock {\em Linear Algebra Appl.}, 139:21--29, 1990.

\bibitem{KalajVuorinen}
David Kalaj and Matti Vuorinen.
\newblock On harmonic functions and the {S}chwarz lemma.
\newblock {\em Proc. Amer. Math. Soc.}, 140(1):161--165, 2012.

\bibitem{Kania}
Tomasz Kania.
\newblock A short proof of the fact that the matrix trace is the expectation of
  the numerical values.
\newblock {\em Amer. Math. Monthly}, 122(8):782--783, 2015.

\bibitem{Mateljevic}
M.~Mateljevi\'{c}.
\newblock Schwarz lemma and {K}obayashi metrics for harmonic and holomorphic
  functions.
\newblock {\em J. Math. Anal. Appl.}, 464(1):78--100, 2018.

\bibitem{Pritsker}
Igor~E. Pritsker.
\newblock Inequalities for integral norms of polynomials via multipliers.
\newblock In {\em Progress in approximation theory and applicable complex
  analysis}, volume 117 of {\em Springer Optim. Appl.}, pages 83--103.
  Springer, Cham, 2017.

\bibitem{RahmanSchmeisser}
Q.~I. Rahman and G.~Schmeisser.
\newblock {\em Analytic theory of polynomials}, volume~26 of {\em London
  Mathematical Society Monographs. New Series}.
\newblock The Clarendon Press, Oxford University Press, Oxford, 2002.

\bibitem{Ramanujan}
S.~Ramanujan.
\newblock Modular equations and approximations to {$\pi$} [{Q}uart. {J}.
  {M}ath. {\bf 45} (1914), 350--372].
\newblock In {\em Collected papers of {S}rinivasa {R}amanujan}, pages 23--39.
  AMS Chelsea Publ., Providence, RI, 2000.

\end{thebibliography}
\bibliographystyle{plain} 

\end{document}